\documentclass[preprint]{imsart}
\usepackage{amssymb}
\usepackage{amsmath}
\usepackage{amsthm}
\RequirePackage{fix-cm}
\usepackage{latexsym}
\usepackage{array}
\usepackage{color}
\usepackage{rotating}
\usepackage{epsfig, wrapfig}
\usepackage{graphicx, lscape, graphics, color} 
\usepackage{verbatim}
\usepackage{latexsym}
\usepackage{fancyvrb}
\usepackage{pdfpages}
\usepackage{fmtcount}
\usepackage{multirow}
\usepackage{url}
\usepackage[english]{babel}
\usepackage[utf8]{inputenc}
\usepackage{tikz} 
\usetikzlibrary{arrows.meta, positioning}
\usetikzlibrary{decorations.markings}
\usepackage{lscape}
\usepackage{epstopdf}
\usepackage{algpseudocode}
\usepackage{algorithm}
\usepackage{kotex}
\usepackage{soul}
\usepackage{enumerate}
\usepackage{colortbl}
\usepackage{booktabs}
\usepackage{appendix}
\usepackage{mathtools}
\usepackage{ulem}

\RequirePackage[numbers,sort&compress]{natbib}

\def\Frechet{Fr\'{e}chet}

\usepackage{amsthm}

\theoremstyle{plain}
\newtheorem{theorem}{Theorem}
\newtheorem{lemma}{Lemma}
\newtheorem{proposition}{Proposition}
\newtheorem{corollary}{Corollary}

\theoremstyle{remark}
\newtheorem{remark}{Remark}

\theoremstyle{definition}
\newtheorem{definition}{Definition}

\usepackage{amssymb}
\usepackage{amsmath}

\begin{document}

\begin{frontmatter}
\title{General M-estimators of location on Riemannian manifolds: existence and uniqueness}
\runauthor{Lee and Jung}
\runtitle{General M-estimators of location on Riemannian manifolds}

\begin{aug}
\author[A]{\fnms{Jongmin}~\snm{Lee}\ead[label=e1]{jongmin.lee@pusan.ac.kr}\orcid{0000-0003-1723-4615}}
\and
\author[B]{\fnms{Sungkyu}~\snm{Jung}\ead[label=e2]{sungkyu@snu.ac.kr}\orcid{0000-0002-6023-8956}}
\address[A]{Department of Statistics,
Pusan National University\printead[presep={,\ }]{e1}}

\address[B]{Department of Statistics and Institute for Data Innovation in Science,
Seoul National University\printead[presep={,\ }]{e2}}
\end{aug}

\begin{abstract}
We study general M-estimators of location on Riemannian manifolds, extending classical notions such as the Fr\'{e}chet mean by replacing the squared loss with a broad class of loss functions. Under minimal regularity conditions on the loss function and the underlying probability distribution, we establish theoretical guarantees for the existence and uniqueness of the associated population M-functional and the corresponding sample M-estimators. In particular, we provide sufficient conditions under which the population minimizer set is nonempty and reduces to a singleton, and under which the corresponding sample M-estimator is likewise uniquely defined. Our results offer a general framework for robust location estimation in non-Euclidean geometric spaces and unify prior uniqueness results under a broad class of convex losses.
\end{abstract}



\end{frontmatter}
 \section{Introduction}
In statistical data analysis, locating the mean of a dataset plays a fundamental role, with applications in data summarization, hypothesis testing, confidence region construction, and analysis of variance. Emerging developments in non-Euclidean data analysis have led to the extension of the classical notion of the \Frechet\ mean to more general M-estimators, enabling it to be a robust measure of central tendency in geometric spaces (see, e.g., \cite{lee2026huber}).

Replacing the $L_{2}$-loss function used in the \Frechet\ mean by a generic loss function $\rho(\cdot)$, we define the general M-estimators of location and the corresponding population parameter characterized by the loss, which we call {\it population M-functional}. 

\begin{definition}\label{def:1}
Let $(M,d)$ be a metric space, and $\mathbb{P}_X$ be a probability measure on $M$. Define the population risk function associated with a loss function $\rho$ as $F(m)=\int \rho\{d(x,m)\}d\mathbb{P}_{X}(x)$. The population M-functional in its set-valued form is the set of minimizers
$$E:=\arg\min_{m \in M} F(m).$$
For an $n$-tuple of observations $(x_1, x_2, \ldots, x_n) \in M^n$, the corresponding sample M-estimator set is 
\[
E_{n}:=\arg\min_{m\in M} F_{n}(m), ~ F_{n}(m)= \frac{1}{n}\sum_{i=1}^{n} \rho\{d(x_i, m)\}.
\]
\end{definition}
In Definition~\ref{def:1}, the minimizers are represented by sets $E$ and $E_n$ rather than single points, since uniqueness is not guaranteed in general.
Theoretical properties of general M-estimators on abstract metric spaces, such as consistency and asymptotic normality, have been the subject of active research in recent years \citep{huckemann2011inference, ginestet2012weighted, sinova2018m, schotz2022strong, park2026strong}.

However, most existing results implicitly assume that the minimizers exist and are unique. Prototypical examples arise in directional statistics, where data lie on the unit sphere and the goal is to estimate a central or ``mean'' direction \citep{mardia2009directional,ley2017modern}. Similarly, on the space of symmetric positive-definite matrices endowed with the affine-invariant Riemannian metric \citep{pennec2006riemannian}, location estimation problems such as barycenters and change-point detection have been extensively studied \citep{jung2025averaging, wang2025riemannian}. In both settings, while classical \Frechet\ means are widely used, their existence, uniqueness, and robustness depend sensitively on the loss function and the underlying geometry. To address this gap, the present article establishes general conditions for the existence and uniqueness of general M-estimators of location, with a particular focus on Riemannian manifolds.


\section{Existence and uniqueness of M-functionals}\label{sec:thms}
Throughout the article, we assume that $M$ is a connected, geodesically complete, and separable $k$-dimensional smooth Riemannian manifold (equipped with a Riemannian metric). 
The Riemannian metric naturally induces the Riemannian distance $d:M \times M \to [0, \infty)$, enabling $(M, d)$ to be a metric space. The least upper bound of sectional curvature of $M$ is denoted by $\Delta$ and assumed to be finite (i.e., $\Delta < \infty$). For details on the Riemannian geometry, see \cite{chavel2006riemannian}.

For an $M$-valued random variable $X$ defined on a probability space equipped with a probability $\mathbb{P}$, the pushforward measure of $\mathbb{P}$ under $X$ is denoted by $\mathbb{P}_{X}$. 
For a point $p \in M$ and $r>0$, we denote $B_{r}(p):=\{x \in M: d(x, p) < r \}$  for an open ball in $M$. For an $n$-tuple of observations $(x_1, x_2, \ldots, x_n) \in M^n$, write the empirical measure $\frac{1}{n}\sum_{i=1}^{n}\delta_{x_i}$ by $\mathbb{P}_{n}$, where $\delta_{x}$ is the Dirac probability measure at $x$. 



We establish conditions that guarantee  existence and uniqueness of the population M-functionals, from which the corresponding results for the sample M-estimator sets follow as corollaries.

A generic loss function $\rho$ is assumed to satisfy the following condition:

\begin{itemize}
    \item[(C1)] $\rho:[0, \infty) \to [0, \infty)$ is non-decreasing and continuous, satisfies $\rho(0)=0$, and is not identically zero.
\end{itemize}

Notably, most standard 
loss functions meet Condition (C1). Examples include the Huber loss $\rho(t) = t^21_{t \le c} + c(2t-c)1_{t>c}$, Tukey's biweight loss $\rho(t)=\frac{c^2}{6}\{1- (1-(\frac{t}{c})^2)^31_{t\le c}\}$, 
Welsch's loss $\rho(t)= 1-\exp(-(\frac{t}{c})^2)$, Cauchy's loss $\rho(t)=\log\{1+(t/c)^2\}$, Andrew's sine loss 
$\rho(t)=c^2(1-\cos(\frac{t}{c}))1_{y \le \pi c} + 2c^21_{t > \pi c}$ and logarithmic hyperbolic cosine loss $\rho(t)=\log\{\cosh(t)\}$. See Table C.5 of \cite{de2021review} for a list of loss functions.

To control the population objective $F$, we impose the following integrability condition for $\mathbb{P}_X$ with respect to $\rho$. 

\begin{itemize}
\item[(A1)] For any $m \in M$, $F(m)=\int \rho\{d(x, m)\}d\mathbb{P}_{X}(x) < \infty$. 
\end{itemize}
To ensure the existence of M-functionals, we first verify a {\it coercivity} property of the objective $F$, meaning that its minimum is attained in a compact set.

\begin{proposition}
Suppose that the loss function $\rho$ satisfies Condition (C1), and that $\mathbb{P}_X$ satisfies Assumption (A1). Then, there exists a compact set $K \subseteq M$ such that 
\begin{eqnarray*}
& \inf_{m \in K} F(m)=\inf_{m \in M} F(m),~ \mbox{and} \\
& \mbox{if} ~ K \subsetneq M, ~ \inf_{m\in K} F(m) < \inf_{m \in M \setminus K} F(m).
\end{eqnarray*}
Accordingly, $E \subset K$.
\label{prop:coercive}
\end{proposition}

The proof of Proposition \ref{prop:coercive} is provided in Appendix. Using the coercivity of $F$, we obtain the following existence result.

\begin{theorem}[Existence of population M-functionals]\label{thm:exist}
Suppose that Condition (C1) holds and that $\mathbb{P}_{X}$ satisfies Assumption (A1). Then, the population M-functionals associated with $\mathbb{P}_{X}$ exists; that is, $E\neq \emptyset$. 
\end{theorem}

Our proof of Theorem \ref{thm:exist} is given in Appendix.

By replacing $\mathbb{P}_{X}$ with the empirical measure $\mathbb{P}_n$ in Theorem \ref{thm:exist}, the existence of the sample M-estimators is established immediately.

\begin{corollary}[Existence of the sample M-estimators]
For any $n$-tuple of observations $(x_{1}, x_2, \ldots, x_n)\in M^n$, the sample M-estimator set  $E_{n}(x_1, x_2, \ldots x_n)$ is non-empty.
\end{corollary}
 
We now turn to conditions ensuring uniqueness, which will rely on one of the following convexity assumptions imposed on the loss function $\rho$. 

\begin{itemize}
\item[(C2)] $\rho'(\cdot)$ is positive and non-decreasing on $(0, \infty)$. Moreover, $\rho''(\cdot)$ is well-defined and is non-negative on $[0, \infty)\setminus S$, where $S$ is a countable set.
\item[(C3)] The right derivative at zero satisfies $\rho'(0):=\lim_{t \to +0}\frac{\rho(t) - \rho(0)}{t}=0$. In addition, $\rho(\cdot)$ is twice continuously differentiable on $[0, \infty)$, and $\rho''(\cdot)$ is positive and non-decreasing on $(0, \infty)$.
\end{itemize}



Condition (C2) is more relaxed than (C3), as the latter implies the former. 
Many standard loss functions, such as the absolute loss and (pseudo-) Huber losses, are convex and satisfy (C2) but not (C3). In contrast, 
strictly convex loss functions such as the $L_{p}$-loss $\rho(t)=t^{p}$ for any $p \in [2, \infty)$ and softplus-square loss $\rho(t)= \{\log(1+e^{t})\}^2$  satisfy (C3). 


We now impose a condition on the support $\textsf{supp}(\mathbb{P}_{X})$ of the underlying distribution, which will interact with the curvature of the manifold to guarantee convexity of the objctive function. Let $r_{\mbox{\tiny inj}}(M)$ denote the injectivity radius of $M$, and recall that $\Delta$ stands for the least upper bound of the sectional curvatures of $M$. By convention, we regard ${1}/{\sqrt{\Delta}}=\infty$ when $\Delta \le 0$.

\begin{itemize}
    \item[(A2)]  There exists $p_{0} \in M$ such that $\textsf{supp}(\mathbb{P}_{X}) \subset B_{r_0}(p_0)$, and 
\begin{equation}\label{eq:r_0}
    r_0 = \begin{cases} \frac{1}{2}\min\{ \frac{\pi}{2\sqrt{\Delta}}, r_{\mbox{\tiny inj}}(M) \} \quad \mbox{if $\rho$ satisfies (C2)}, \\
    \frac{1}{2}\min\{ \frac{\pi}{\sqrt{\Delta}}, r_{\mbox{\tiny inj}}(M) \}
\quad \mbox{if $\rho$ satisfies (C3)}.    
    \end{cases}
\end{equation} 
\end{itemize}

When $M$ is Hadamard, $r_0$ specified in (\ref{eq:r_0}) is infinite; thus, in this case, Assumption (A2) is not a restriction on $\mathbb{P}_{X}$ at all.

\begin{theorem}[Uniqueness of population M-functional]\label{thm:unique:pop}

Suppose Condition (C1) holds and $\mathbb{P}_{X}$ satisfies Assumption (A1). The population M-functional associated with $\mathbb{P}_X$ is unique if either of the following conditions holds:
\begin{enumerate}
\item[(a)] $\rho$ satisfies Condition (C2), $\mathbb{P}_{X}$ satisfies Assumption (A2), and $\mathbb{P}_X$ is not supported entirely on any single geodesic;

\item[(b)] $\rho$ satisfies Condition (C3), and $\mathbb{P}_{X}$ satisfies Assumptions (A2).
\end{enumerate}
\end{theorem}
Our proof of Theorem~\ref{thm:unique:pop}, as well as some auxiliary lemmas, are given in Appendix. 
%
%
By replacing $\mathbb{P}_{X}$ with $\mathbb{P}_{n}$, the uniqueness of the sample M-estimator follows.
\begin{corollary}[Uniqueness of sample M-estimators]\label{cor:unique:sam}
Suppose Condition (C1) holds and that either of the following holds: 
\begin{enumerate}
\item[(a)] $\rho$ satisfies (C2), and the sample points $\{x_{1}, x_{2}, ..., x_{n}\}$ do not lie in the same geodesic, and reside in an open ball of radius $r_0 = \frac{1}{2}\min\{\frac{\pi}{2\sqrt{\Delta}}, r_{\mbox{\tiny inj}}(M) \}$. 

\item[(b)] $\rho$ satisfies (C3), and the sample points $\{x_{1}, x_{2}, ..., x_{n}\}$ lie in an open ball of radius $r_0 = \frac{1}{2}\min\{\frac{\pi}{\sqrt{\Delta}}, r_{\mbox{\tiny inj}}(M) \}$. 
\end{enumerate}
Then, the sample M-estimator is unique. 
\end{corollary}

\begin{remark}
    \begin{enumerate}
    \item[(a)] The conclusion in Theorem~\ref{thm:unique:pop}(a) extends Theorem~2(a)(i) of \cite{lee2026huber}, while Theorem~\ref{thm:unique:pop}(b) generalizes Theorem~2.1 of \cite{afsari2011riemannian} for the $L_{p}$-loss function $\rho(t)=t^p$ with $2 \le p < \infty$.
    
    \item[(b)] In Theorem \ref{thm:unique:pop} and Corollary \ref{cor:unique:sam}, the imposed convexity assumption on $\rho$ plays a crucial role in guaranteeing uniqueness. In contrast, establishing uniqueness becomes substantially more challenging when $\rho$ is non-convex. Indeed, it is well-known that M-estimators with non-convex (e.g., redescending) losses may fail to be unique even in Euclidean spaces; see, for example, \citep[][Chapter 4.8]{huber2009robust}.
\end{enumerate}
\end{remark}

\section{Conclusion}

We have established the existence and uniqueness of generic M-estimators on Riemannian manifolds under mild regularity conditions on the loss function $\rho$ and suitable assumptions on the underlying probability measure $\mathbb{P}_{X}$. These results provide a foundational basis for the development of methodologies and theories for manifold-valued M-estimation. For instance, our framework offers a theoretical starting point for comparing different M-estimators in terms of efficiency and robustness. More broadly, it can be used to justify robust dimension reduction techniques for manifold-valued data (e.g., \citealp{jung2012analysis}) and to support the construction of robust regression framework replacing the \Frechet-type functionals used in \cite{petersen2019frechet}.


\section*{Acknowledgements}
Jongmin Lee was supported by the Research Fund Program of Institute for Future Earth, Pusan National University, Korea, 2023, Project No. IFE-PNU-2023-001. Sungkyu Jung was supported by the National Research Foundation of Korea (NRF) grants funded by the Korea government (MSIT) (RS-2023-00301976 and RS-2024-00333399).




\begin{appendix}
\section{Proof of Proposition 1}
\begin{proof}[Proof of Proposition 1]
The proof is divided into two cases depending on whether $M$ is bounded or not.

\indent {\bf Case I ($M$ is bounded).} Put $K=M$. Note that $M$ is connected and complete. Since $K$ is closed and bounded, an application of the Hopf-Rinow theorem implies that $K$ is compact and $\inf_{m \in K} F(m)=\inf_{m \in M} F(m)$.

\indent {\bf Case II ($M$ is unbounded).} We shall show that $\{F(m): m \in M\}$ is not singleton. For this, suppose for contradiction there exists $C \in [0, \infty)$ such that $F(m)=C$ for any $m \in M$. Let $\sup_{t \in [0, \infty)}\rho(t) = \lim_{t \to \infty} \rho(t) =: U \in (0, \infty]$ and pick a point $p \in M$. Since $M$ is unbounded, there exists an escaping geodesic ray $\gamma:[0, \infty) \to M$ starting at $p=\gamma(0)$ (see, e.g., \citep[][Exercise 7.6]{do1992riemannian}). The term \textit{escaping geodesic} means that $d(p, \gamma(t)) \underset{t\to\infty}{\to} \infty$; that is, the escaping geodesic extends infinitely in one direction and does not oscillate wildly. 
An application of Fatou's lemma leads that
\begin{equation}
    \label{eq:coer1}
\mathbb{E} [\liminf_{t\to \infty} \rho\{d(X, \gamma(t))\}] \le \liminf_{t\to\infty}\underbrace{\mathbb{E}[\rho\{d(X,\gamma(t))\}]}_{=C} \le U.
\end{equation} 
Furthermore, by the triangle inequality, we have for any $x \in M$, 
\begin{equation}
    \label{eq:coer2}
U \ge \liminf_{t \to \infty} \rho\{d(x, \gamma(t))\} \ge \liminf_{t \to \infty} \rho\{d(p, \gamma(t))- d(x, p)\} = U.
\end{equation} 

If $U = \infty$, then $\mathbb{E} [\liminf_{t\to \infty} \rho\{d(X, \gamma(t))\}] = \infty$, which contradicts to the hypothesis of finite $C$. 
Thus, if $U = \infty$, then $\{F(m): m \in M\}$ is not singleton. 

For the case $U <\infty$, (\ref{eq:coer1}) and (\ref{eq:coer2}) together lead that $F(m) = C = U \in (0,\infty)$. 
 Owing to the continuity of $t\mapsto \rho(t)$ at $t=0$ with $\rho(0)=0$ (inherent in Condition (C1)), there exists $\delta > 0$ such that $0 \le t < \delta$ implies $\rho(t) < U/2$. Since $M$ is separable, there is a countable dense subset of $M$, $\{p_1, p_2, \ldots \}$. Note that $\cup_{i=1}^{\infty} B_{\delta}(p_i) = M$ and thus there exists at least one point $p_{\ell}$ such that $\mathbb{P}_{X}(B_{\delta}(p_{\ell})) = \mathbb{P}(X \in B_{\delta}(p_{\ell})) > 0$ (since if not, it holds that $1=\mathbb{P}_{X}(M)=\mathbb{P}_{X}(\cup_{i=1}^{\infty}B_{\delta}(p_i))\le \sum_{i=1}^{n}\mathbb{P}_{X}(B_{\delta}(p_i))=0$ due to the countable sub-additivity of probability measure, which is a contradiction). Hence, it holds that

\begin{eqnarray*}
F(p_{\ell}) &=& \mathbb{E}[\rho\{d(X, p_{\ell})\}1_{X \notin B_{\delta}(p_{\ell})}] + \mathbb{E}[\rho\{d(X, p_{\ell})\}1_{X \in B_{\delta}(p_{\ell})}]\\
&\le&U\cdot \mathbb{P}(X \notin B_{\delta}(p_{\ell})) + U/2 \cdot \mathbb{P}(X \in B_{\delta}(p_{\ell})) \\
&=& U \{ \mathbb{P}(X \notin B_{\delta}(p_{\ell})) + \tfrac{1}{2} \mathbb{P}(X \in B_{\delta}(p_{\ell}))\}\\
&<& U \{ \mathbb{P}(X \notin B_{\delta}(p_{\ell})) + \mathbb{P}(X \in B_{\delta}(p_{\ell}))\}\\
&=& U,
\end{eqnarray*}
which contradicts to the hypothesis that $F(m)=U$ for any $m \in M$. Therefore, $\{F(m): m \in M\}$ is not a singleton set.

The remainder of the proof is divided into two cases depending on whether $U$ is infinite or not. 

\underline{(i) $U= \infty$.} Because $\{F(m): m\in M\}$ is not a singleton set (thus, it has at least two elements), there exists $m_1 \in M$ such that $F(m_1) > \inf_{m\in M} F(m) \ge 0$. Because $M=\cup_{r \ge 0}B_{r}(m_1)$ and $\mathbb{P}(X \in M) = 1$, there exists $r_0 > 0$ such that $\mathbb{P}(X \in B_{r_0}(m_1)) > 0$. Now put $K_0 = \overline{B_{r_0}(m_1)}$. 
Recall that $\rho$ is non-decreasing as assumed in (C1) and satisfies $\lim_{t\to \infty}\rho(t)=\infty$. Thus, for any finite $u$, there exists a $t_0>0$ such that, for all $t\ge t_0$, $\rho(t) > u$. Accordingly, there exists $r_1 ~(> r_0)$ such that for any $r \ge r_1$
\[
\rho(r-r_0)> \frac{\inf_{m \in K_0} F(m)}{\mathbb{P}(X \in K_0)}.
\] 
We set $K:= \overline{B_{r_1}(m_1)}$, which contains $K_0$ and is compact by the Hopf-Rinow theorem. 
For each $m \in M \setminus K$,
\begin{equation}\label{lem:coercive:eq1}
F(m) \ge \int_{K_0} \rho\{d(x, m)\}d\mathbb{P}_{X}(x) \ge \rho(r_1 - r_0)\mathbb{P}(X \in K_0) > \inf_{m \in K_0}F(m).
\end{equation}
Here, the second inequality holds by $d(x, m) \ge d(m_1, m) - d(m_1, x) \ge r_1 -r_0$ for any $x \in K_0$. From (\ref{lem:coercive:eq1}), we get
\[
 \inf_{m \in M\setminus K} F(m) \ge \rho(r_1 - r_0)\mathbb{P}(X \in K_0) > \inf_{m \in K_0}F(m) \ge \inf_{m \in K} F(m) = \inf_{m\in M}F(m),
\]
thereby meaning that $E \subset K$. 

\underline{(ii) $U < \infty$.}
Since $\{F(m): m \in M\}$ is not singleton, there is a point $m_0 \in M$ such that $F(m_0) < U$. Note that for any $\epsilon > 0$, there exists $r(\epsilon)>0$ such that $r \ge r(\epsilon)$ implies $\mathbb{P}(X \in B_{r}(m_0)) \ge 1-\epsilon$ and there also exists $t(\epsilon) >0$ such that $t > t(\epsilon)$ implies $\rho(t) \ge U-\epsilon$. For any $m \in M\setminus B_{r(\epsilon)+t(\epsilon)}(m_0)$, 

\begin{equation}\label{ineq:coercive}
F(m)\ge  \int_{B_{r(\epsilon)}(m_0)}\rho\{d(x, m)\}d\mathbb{P}_{X}(x)   
\ge (1-\epsilon)(U-\epsilon),  
\end{equation}
where the last inequality holds because for any $x \in B_{r(\epsilon)}(m_0)$, $d(x,m) \ge d(m,m_0) - d(x, m_0) >r(\epsilon) +t(\epsilon) -r(\epsilon) = t(\epsilon)$; thus, $\rho\{d(x, m)\} \ge U-\epsilon$. Notably, the inequality (\ref{ineq:coercive}) holds for any $\epsilon > 0$. Since the quantity in (\ref{ineq:coercive}), $(1-\epsilon)(U-\epsilon)$, can be arbitrarily close to 
$U$ for infinitesimal $\epsilon$, and since $U > F(m_0)$, we have, 
for sufficiently small $\epsilon$, 
$U > (1-\epsilon)(U-\epsilon) > (U+F(m_0))/2 > F(m_0).$
According to this fact and (\ref{ineq:coercive}), there exists $\epsilon_0 > 0$ such that for any $m \in M\setminus B_{r(\epsilon_0)+t(\epsilon_0)}(m_0)$, 
\[
F(m) \ge \frac{U + F(m_0)}{2} > F(m_0).
\]
Letting $K= \overline{B_{r(\epsilon_0)+t(\epsilon_0)}(m_0)}$ gives 
\[
\inf_{m \in M\setminus K}F(m) \ge \frac{U+F(m_0)}{2} > F(m_0) \ge \inf_{m \in K}F(m) = \inf_{m\in M}F(m).
\]

Since $K$ is closed and bounded, it is compact by the Hopf-Rinow theorem, and the proof is complete.
\end{proof}

\section{Proof of Theorem 1} 
\begin{proof}[Proof of Theorem 1]
For any $m \in M$ and for any sequence of points $(m_{n})_{n\ge 1}$ converging to $m$, an application of Fatou's lemma leads that
\begin{eqnarray*}
\liminf_{n \to \infty} F(m_n) = \liminf_{n\to \infty}\mathbb{E}[\rho\{d(X,m_n)\}] &\ge& \mathbb{E}[\liminf_{n \to \infty}\rho\{d(X, m_n)\}] \\
&=& \mathbb{E}[\rho\{d(X,m)\}] \\
&=& F(m),
\end{eqnarray*}
where the penultimate equality holds by the continuity of $\rho$. Accordingly, the map $m \mapsto F(m)$ is lower semi-continuous. Let $K$ be the compact set from Proposition 1. Because $F$ is lower semi-continuous, it attains its minimum on the compact set $K$. Since the population M-functional set $E$ includes such a minimum, it is not empty.    
\end{proof}

\section{Auxiliary lemmas and a proof of Theorem 2}

To verify Theorem 2, we begin with a lemma which ensures that the population M-functional set associated with $\mathbb{P}_{X}$ lies in a strongly convex ball. 

\begin{lemma}\label{lem:contain:ball}
Suppose that Conditions (C1) and either (C2) or (C3) hold, and that $\mathbb{P}_{X}$ satisfies Assumptions (A1) and (A2). Then, there exist $p_0 \in M$ and a finite $r \in (0,r_0)$ 
such that $E \subset B_r(p_0)$ and the ball $B_{r}(p_0)$ is strongly convex.
\end{lemma}

\begin{proof}[Proof of Lemma \ref{lem:contain:ball}]

  The proof is divided into two cases, depending on the value of $r_0$: Case I for $r_0 = \infty$ and Case II for $r_0 <\infty$.

{\bf Case I ($r_0 = \infty$).} In this case, Assumption (A2) is of limited utility, as it allows for the unbounded support for $\mathbb{P}_{X}$. On the other hand, since the objective function is integrable, there exists a compact set $K\subset M$ such that $E \subset K = \overline{B_{r'}(p_0)}$ for some $r' <\infty$ and $p_0 \in M$. (See the proof of Proposition 1 in Appendix.) Pick a small $\epsilon > 0$, and let $r = r' + \epsilon$. Then, $E \subset K \subset B_r(p_0)$, as asserted. 

{\bf Case II ($r_0 < \infty$).} 
Since $\textsf{supp}(\mathbb{P}_{X})$ is closed, Assumption (A2) leads that there exists $\epsilon>0$ satisfying $\textsf{supp}(\mathbb{P}_{X}) \subset B_{r_0 - \epsilon}(p_0) \subset  B_{r_0}(p_0)$, where $r_0, p_0$ are specified in Assumption (A2). Let $r = r_0 - \epsilon$.

By definition, the regular convexity radius of $M$, denoted by $r_{\mbox{\tiny cx}}(M)$, is the supremum of radii such that geodesic balls of that radius are strongly convex. It is also known that $r_{\mbox{\tiny cx}}(M)$ is lower bounded by $\frac{1}{2} \min\{\frac{\pi}{\sqrt{\Delta}}, r_{\mbox{\tiny inj}}(M)\}$ \citep[][pp. 404]{chavel2006riemannian}. Since $r \in (0, r_0)$, the ball $B_r(p_0)$ is strongly convex.

We now show that $E \in B_r(p_0)$, which closely follows the argument used in the proof for Theorem 2.1 of \cite{afsari2011riemannian}.

Firstly, we show that $E \subset B_{2r}(p_0)$. Suppose to the contrary that 
there exists $m_0 \in E$ such that $d(p_0, m_0) \ge 2r$. By 
the triangle inequality, for any $x \in \textsf{supp}(\mathbb{P}_{X})$, we have $d(p_0, x) < r \le d(x, m_0)$, which implies that  $\rho\{d(p_0, x)\} < \rho\{d(x, m_0)\}$. 
Therefore, $F(p_0) < F(m_0)$, contradicting the assumption that $m_0 \in E$. 

Secondly, we show that $E \subset \overline{B_{r}(p_0)}$. Suppose, for contradiction, that there exists $m_0 \in E$ such that $r < d(m_0, p_0) < 2r$. Let $m'$ be the intersection of the (length-minimizing) geodesic, joining $p_{0}$ and $m_0$, and the boundary of $B_{r}(p_0)$. Let $m''$ be the reflection point of $m_0$ with respect to $m'$ across the boundary of $B_{r}(p_0)$. Then, $m'' \in B_{r}(p_0)$ and $d(m_0, m')=d(m', m'')$. As shown in Section 2.2.1 of \cite{afsari2011riemannian}, for any $x \in \textsf{supp}(\mathbb{P}_{X}) \subseteq B_{r}(p_0)$, $d(m'', x) < d(m_{0}, x)$. 
(This result relies on an application of the Toponogov comparison theorem; see \cite[][pp. 420]{chavel2006riemannian} 
or \cite[][Theorem 2.2]{afsari2011riemannian}.) Therefore, $\rho\{d(m'', x)\} < \rho\{d(m_{0}, x)\}$ for any $x \in \textsf{supp}(\mathbb{P}_{X})$, and thus $F(m'') < F(m_{0})$, contradicting the assumption that $m_0\in E$. 

Thirdly, we verify that $E \subset B_{r}(p_0)$. Suppose, again for contradiction, that there exists $m_0 \in E$ such that $m_0$ lies on the boundary of $B_{r}(p_0)$. For each $x \in B_{r}(p_0)$, either by Condition (C2) or by Lemma \ref{lem:ineq} under (C3), $\rho'\{d(x, m_0)\} >0$ and 
the tangent vector $\mbox{Log}_{m_{0}}(x)/\|\mbox{Log}_{m_{0}}(x)\|_{m_0}$ 
points inward at the boundary of $B_{r}(p_0)$. 
Hence, $-\textsf{grad}F(m_0)= \int \rho^{'}\{d(x, m_0)\} \cdot \mbox{Log}_{m_0}(x)/\|\mbox{Log}_{m_0}(x)\|_{m_0}d\mathbb{P}_{X}(x) \neq \mathbf{0}$ 
is an inward-pointing and non-zero tangent vector at $m_0$. Therefore, $m_0 \notin E$, contradicting the hypothesis of the existence of an M-estimator at the boundary of $B_r(p_0)$. 
\end{proof}

Next technical lemma is used as well. 

\begin{lemma}\label{lem:ineq}
Under Condition (C3), $\rho'$ is strictly increasing. Moreover, for any $t \in [0, \infty)$, $\rho'(t) \le t\rho''(t)$.
\end{lemma}

\begin{proof}[Proof of Lemma \ref{lem:ineq}]
$\rho''>0$ implies that $\rho'$ is strictly increasing. The fundamental theorem of calculus implies that for any $t \in [0, \infty)$,
\begin{eqnarray*}
\rho'(t)= \rho'(0) + \int_{0}^{t} \rho''(u)du \le \int_{0}^t \rho''(t)du=t\rho''(t),
\end{eqnarray*}
where above inequality holds since $\rho''$ is non-decreasing.
\end{proof}

We are now ready to prove Theorem~2. 

\begin{proof}[Proof of Theorem~2]

Suppose, for contradiction, that the population M-functional associated with $\mathbb{P}_{X}$ is not unique. Then, there exist at least two distinct population M-functionals, denoted by $m_{0}$ and $m'_{0}$, associated with $\mathbb{P}_X$. Note that both   $m_{0}$ and $m'_{0}$ are 
minimizers of the objective function $F(\cdot)$, and by Lemma \ref{lem:contain:ball}, they lie in the strongly convex $B_r(p_0)$ (for $r,p_0$ specified in Lemma \ref{lem:contain:ball}). Thus, there exists a unique geodesic $\gamma$ with unit speed connecting $m_{0} ~(= \gamma(0))$ and $m'_{0} ~(=\gamma(t_0))$ so that for some $-\infty < t_1 < t_2 < \infty$ satisfying $ t_1 < 0 < t_0 < t_2$ and $\gamma((t_1, t_2)) = \gamma \cap B_r(p_0)$.
%
%
(We abuse the notation for the geodesic $\gamma$, which either refers to a function from an interval to $M$ or the image of the function $\gamma$ in $M$.)

We now begin by verifying case (a). 

(a) First of all, Assumption (C2) leads that, for any $x \in B_{r}(p_0)$ and 
for any $t \in (t_1,t_2)$,  $\tfrac{\partial^2}{\partial t^2} d(x,\gamma(t))$ is well-defined for $x \neq \gamma(t)$. 
We further claim that for 
$ x \in \textsf{supp}(\mathbb{P}_X) \setminus \gamma$
and for any $t \in (t_1,t_2)$, the following holds: 
\begin{equation}\label{eq:x_notin_gamma}
   \tfrac{\partial^2}{\partial t^2} d(x,\gamma(t)) > 0.
\end{equation}
This can be seen from the Hessian comparison theorem, which states that 
for any $x \in \textsf{supp}(\mathbb{P}_X)$ and for any $t \in (t_1,t_2)$,
\begin{equation}
    \label{eq:HessianComparison} \tfrac{\partial^2}{\partial t^2} d(x,\gamma(t)) \ge \frac{\textsf{sn}_{\Delta}'(d(x,\gamma(t))}{\textsf{sn}_{\Delta}(d(x,\gamma(t))} \sin^2(\alpha),
\end{equation} 
where $\alpha$ is the angle formed at $\gamma(t)$ by the geodesic $\gamma$ and the unique length-minimizing geodesic from $x$ to $\gamma(t)$, and 
$\textsf{sn}_{\Delta}(s) = \tfrac{1}{\sqrt{\Delta} }\sin (s\sqrt{\Delta} )$ for $\Delta >0$, 
$s$ for $\Delta =0$, and $\tfrac{1}{\sqrt{|\Delta|} }\sinh (s\sqrt{|\Delta|})$ for $\Delta <0$ \citep{afsari2011riemannian}. 
Employing the triangle inequality and $r < r_0 \le \frac{\pi}{4\sqrt{\Delta}}$ as assumed in (A2), we obtain for any $x \in \textsf{supp}(\mathbb{P}_X) \subset B_{r}(p_0)$, $d(x,\gamma(t)) < \tfrac{\pi}{2\sqrt{\Delta}}$, which in turn leads that 
$$
\frac{\textsf{sn}_{\Delta}'(d(x,\gamma(t))}{\textsf{sn}_{\Delta}(d(x,\gamma(t))} > 0, 
$$
unless $d(x,\gamma(t)) = 0$. 
The above result and the fact that for $x \in B_{r}(p_0) \setminus \gamma$, the angle is non-zero (i.e., $\sin(\alpha) \neq 0$), together with (\ref{eq:HessianComparison}), lead to (\ref{eq:x_notin_gamma}). On the other hand, if $x \in \gamma \cap \textsf{supp}(\mathbb{P}_X)$, but $x \neq \gamma(t)$, then 
$\sin(\alpha) = 0$, leading that 
$\tfrac{\partial^2}{\partial t^2} d(x,\gamma(t)) = 0$.

The following holds for all $t \in (t_1, t_2)$ satisfying that there is no point mass at $\gamma(t)$ or at any point $y \in B_{r_0}(p_0)$ such that $d(y,\gamma(t)) \in S$:
%
\begin{equation}
\tfrac{\partial^2 }{\partial t^2} F(\gamma(t)) = \int \rho''\{d(x, \gamma(t))\} \{\tfrac{\partial }{\partial t}d(x,\gamma(t))\}^2  + \rho'\{d(x, \gamma(t))\} \tfrac{\partial^2}{\partial t^2}d(x, \gamma(t))d\mathbb{P}_{X}(x). 
\label{eq:2ndderivative_2}
\end{equation}

Since both integrands are nonnegative, we have 
\begin{align}\label{eq:2nd_deriv_F}
\tfrac{\partial^2 }{\partial t^2} F(\gamma(t))     \ge& \int_{\gamma} \rho''\{d(x,\gamma(t))\} \{\frac{\partial}{\partial t}d(x,\gamma(t))\}^2 d\mathbb{P}_{X}(x) \nonumber \\
    &+ \int_{B_r(p_0) \setminus \gamma} \rho'\{d(x, \gamma(t))\} \frac{\partial^2}{\partial t^2}d(x, \gamma(t)) d\mathbb{P}_{X}(x)  > 0.  
\end{align}
Note that (\ref{eq:2nd_deriv_F}) is positive wherever it is defined because $\mathbb{P}_{X}$ is not supported entirely on $\gamma$ by assumption and for $x\in \textsf{supp}(\mathbb{P}_X)\setminus \gamma$, the second integrand of (\ref{eq:2nd_deriv_F}) is positive with non-zero mass. 

Since there are at most countable point masses along $\gamma$ and even in $B_{r_0}(p_0)$ (together with the fact that $S$ is countable), (\ref{eq:2nd_deriv_F}) is well-defined and is positive for $t \in (t_1, t_2)$, possibly excluding a countable set. Hence,
$t \mapsto \frac{\partial}{\partial t}F(\gamma(t))$ is strictly increasing by the Goldowsky-Tonelli theorem (see, e.g., Theorem 1 of  \cite{casey2005positive}). It implies that $t \mapsto F(\gamma(t))$ is strictly convex, which is a contradiction to the hypothesis of two distinct population M-functionals associated with $\mathbb{P}_{X}$.

(b) We next show that if the condition (b) holds, then the population M-functional associated with $\mathbb{P}_{X}$ is unique. 
Suppose again that, for contradiction, $m_0$ and $m_0'$ are both M-functionals and $\gamma$ is the unique geodesic connecting $m_0$ and $m_0'$. 

We first note that Assumption (C3) also leads that $\tfrac{\partial^2}{\partial t^2} d(x,\gamma(t))$ is well-defined for $x \in B_r(p_0)$ and $t \in (t_1,t_2)$ as long as $x \neq \gamma(t)$. The Hessian comparison  (\ref{eq:HessianComparison}) is still valid for this case as well. Moreover, it can be seen that for $\Delta \le 0$, 
$\frac{\textsf{sn}_{\Delta}'(s)}{\textsf{sn}_{\Delta}(s)} > 0$ for any $s>0$, but for $\Delta > 0$, 
$\frac{\textsf{sn}_{\Delta}'(s)}{\textsf{sn}_{\Delta}(s)}$ is nonnegative only for $s \le \pi/2$ and is negative  for $s > \pi/2$. Accordingly, we divide the proof into two cases. 


{\bf Case I ($\Delta \le 0$).}  
 From 
 (\ref{eq:HessianComparison}), we obtain that for any $t \in (t_1,t_2)$, 
 $\tfrac{\partial^2}{\partial t^2} d(x,\gamma(t))$ is positive for $x \in \mathsf{supp}(\mathbb{P}_X) \setminus \gamma$, and is zero for $x \in \gamma \cap \mathsf{supp}(\mathbb{P}_X)$. Therefore, the same argument used in part (a) completes the proof.

{\bf Case II ($\Delta > 0$).} 
Since $\Delta > 0$, $r_0$ specified in Assumption (A2) is finite. Let $B_r(p_0)$ with $r ~(< r_0)$ be as in the statement of Lemma 2, so that both $\textsf{supp}(\mathbb{P}_X)$ and $E$ are contained in the strongly convex ball $B_r(p_0)$. In this case, the inequality $\frac{\partial^2}{\partial t^2}d(x, \gamma(t)) > 0$ for $x \in B_{r}(p_0)\setminus \gamma$ is not guaranteed: the Hessian comparison (\ref{eq:HessianComparison}) combined with the assumption $r_0\le \frac{\pi}{2\sqrt{\Delta}}$ in (A2) does not ensure the strict positivity. Therefore, the proof requires a more refined argument.

If $\mathbb{P}_{X}$ is supported entirely on the geodesic $\gamma$, then $\frac{\partial^2}{\partial t^2}d(x, \gamma(t)) = 0$ for any $x \in \textsf{supp}(\mathbb{P}_X)$ and $t \in (t_1,t_2)$. This fact, together with Assumption (C3) leads that $\frac{\partial^2 }{\partial t^2}F(\gamma(t))$ is strictly positive, as needed. 

Hence, in the rest of the proof we assume that  $\mathbb{P}_{X}$ is not entirely supported on any single geodesic contained in $B_{r}(p_0)$.

Since $E ~(\subset B_r(p_0))$ is non-empty by Theorem 1, there is at least one critical point of $F$ in $B_r(p_0)$, which is in fact a minimizer. While the existence of other critical points that are not minimizers can not be ruled out a priori, we claim that any critical point of $F$, restricted to $B_r(p_0)$, is in fact a local minimizer. If this claim holds, then the Poincar\'{e}-Hopf index theorem \citep[][Chapter 3.5]{guillemin2010differential} applied to the compact submanifold $\overline{B_{r}(p_0)}$ of $M$ leads that $F$ has exactly one critical point in $\overline{B_{r}(p_0)}$, which must be the global minimizer. Thus, $E$ is a singleton, completing the proof.

It is left to verify the claim that 
any critical point of $F$, restricted to $B_r(p_0)$, is in fact a local minimizer. Let $m_0\in B_{r}(p_0)$ be such a critical point, and let $\eta$ be any unit-speed geodesic with $\eta(0) = m_0$.

%
%
By applying the Hessian comparison theorem (\ref{eq:HessianComparison}), we obtain for $x \in B_{r}(p_0)\setminus \{m_0\}$,
$\tfrac{\partial^2}{\partial t^2}\mid_{t=0} d(x,\eta(t)) \ge \sqrt{\Delta}\cot\{\sqrt{\Delta}d(x, m_0) \} \cdot \sin^2(\alpha_{x}),
$
where $\alpha_{x}$ denotes the angle formed at $m_0$ between the geodesic segment from $m_0$ to $x$ and $\eta$.

Since $\eta$ has a unit speed, furthermore, $\frac{\partial}{\partial t}\mid_{t=0}d(x, \eta(t)) = -\cos(\alpha_{x})$ for $x \neq m_0$. 
For $x = m_0$, since $\eta$ is a unit speed geodesic, $d(m_0, \eta(t)) = |t|$, and $\rho\{d(m_0, \eta(t))\} = \rho(|t|)$. By Condition (C3), the function $\rho_0: \mathbb{R}\to\mathbb{R}$, $\rho_0(t):= \rho(|t|)$, is twice continuously differentiable at 0, and $\tfrac{\partial^2}{\partial t^2}\mid_{t=0}\rho\{d(m_0,\eta(t))\} = \rho_0''(0) = \rho''(0) \ge 0$.
These facts, together with (\ref{eq:2ndderivative_2}), 
lead that 
\begin{equation}\label{eq:2ndderiv_bound_decomp1}
        \frac{\partial^2}{\partial t^2}\mid_{t=0}F(\eta(t)) \ge I_1(m_0,\eta) + I_2(m_0,\eta),
\end{equation}
where 
\begin{align}
I_1(m_0,\eta)  = & \int_{B_{r}(p_0)\setminus \{m_0\}} \big[\rho''\{d(x, m_0)\}\cos^2(\alpha_{x}) \nonumber \\
&+ \rho'\{d(x, m_0)\} \cdot \sqrt{\Delta}\cot\{\sqrt{\Delta}d(x, m_0) \} \cdot \sin^2(\alpha_{x}) \big]d\mathbb{P}_{X}(x), \label{eq:2ndderiv_bound_decomp2}
\end{align}
and 
$I_2(m_0,\eta) = \int_{\{m_0\}}\frac{\partial^2}{\partial t^2}\mid_{t=0}\rho\{d(x,\eta(t))\}d\mathbb{P}_{X}(x) = \rho''(0)\mathbb{P}_{X}(\{m_0\}) \ge 0$.

Suppose that $m_0=p_0$. Then, the integrand of (\ref{eq:2ndderiv_bound_decomp2}) is positive for any $x \in B_{r}(m_0)\setminus \{m_0\}$. This can be seen from the fact that $d(x, m_0) < r \le 
{\pi}/({2\sqrt{\Delta}})$, thus $\cot\{\sqrt{\Delta}d(x, m_0) \} > 0$, and by Condition (C3), in particular, that $\rho''(\cdot), \rho'(\cdot) > 0$ on $(0, \infty)$ with $\rho''(0) \ge 0$. Therefore, (\ref{eq:2ndderiv_bound_decomp1}) is positive unless $\textsf{supp}(\mathbb{P}_{X})=\{m_0\}$, which is precluded by the assumption that $\mathbb{P}_{X}$ is not supported entirely on any geodesic. Since $\eta$ was arbitrary, the Hessian of $F$ at $m_0$ is positive definite; hence, $m_0$ is a local minimizer of $F$, as desired.

Hereafter, suppose $m_0 \neq p_0$.  
Plugging-in the inequality $\sqrt{\Delta}d(x, m_0) \cdot \cot\{\sqrt{\Delta}d(x, m_0) \} \le 1$ into the integrand of (\ref{eq:2ndderiv_bound_decomp2}) and utilizing $\rho'(t) \le t \rho''(t)$ from Lemma \ref{lem:ineq}, we obtain
\begin{align}\label{ineq:2nd_deriv2}
I_1(m_0,\eta) \ge &\int_B \big[\rho''\{d(x, m_0)\}d(x, m_0) \cdot \sqrt{\Delta} \cdot \cot\{\sqrt{\Delta}d(x, m_0)\} \cos^2(\alpha_{x}) \nonumber \\
&+ \rho'\{d(x, m_0\}\sqrt{\Delta}\cdot \cot\{\sqrt{\Delta}d(x, m_0)\}\sin^2(\alpha_{x})\big]d\mathbb{P}_{X}(x) \nonumber \\
\ge&  \int_B \big[\rho'\{d(x,m_0)\}\sqrt{\Delta}\cdot \cot\{\sqrt{\Delta}d(x, m_0)\}\cos^2(\alpha_{x}) \nonumber \\
&+  \rho'\{d(x, m_0\}\sqrt{\Delta}\cdot \cot\{\sqrt{\Delta}d(x, m_0)\}\sin^2(\alpha_{x}) \big]d\mathbb{P}_{X}(x) \nonumber \\
=& \int_B \rho'\{d(x, m_0)\}\sqrt{\Delta}\cdot \cot\{\sqrt{\Delta}d(x, m_0)\} d\mathbb{P}_{X}(x), 
\end{align}
where $B := B_r(p_0) \setminus \{m_0\}$.

For $x \in B$, let $\beta_{x}$ denote the angle between the geodesic segment from $m_0$ to $x$ and the geodesic segment from $m_0$ to $p_0$. Utilizing the Toponogov comparison theorem in conjunction with the spherical geometry, an argument similar to Section 2.2.3 of \cite{afsari2011riemannian} leads that 
for any $x \in B$,
\begin{align}\label{ineq:central:positive}
\cot\{\sqrt{\Delta}d(x,m_0)\} > C(m_0, \Delta)\cos(\beta_{x}),
\end{align}
where $C(m_0, \Delta)$ denotes some constant depending solely on $m_0$ and $\Delta$. (This same argument was used in verifying uniqueness of the $L_p$ center of mass, $2 \le p < \infty$, in Theorem 2.1 of \cite{afsari2011riemannian}.)
Combining (\ref{eq:2ndderiv_bound_decomp1}) and (\ref{ineq:2nd_deriv2}) with (\ref{ineq:central:positive}), we have 
\begin{align}
    \frac{\partial^2}{\partial t^2}\mid_{t=0}F(\eta(t)) & > \int_B \rho'\{d(x,m_0)\} \sqrt{\Delta} C(m_0, \Delta)\cos(\beta_{x})d\mathbb{P}_{X}(x) \nonumber \\
    & = \sqrt{\Delta} C(m_0, \Delta) \int \rho'\{d(x,m_0)\} \cos(\beta_{x})d\mathbb{P}_{X}(x). \label{eq:2ndderiv_final}
\end{align} 
The range of integration for the last expression is $B_r(p_0)$, since $\rho'\{d(x,m_0)\} = 0$ at $x = m_0$. 

Since $m_0$ is a critical point, it holds that $$\textsf{grad}F(m_0)= -\int\rho'\{d(x, m_0)\} \cdot \mbox{Log}_{m_0}(x)/\|\mbox{Log}_{m_0}(x)\|_{m_0}d\mathbb{P}_{X}(x) = \textbf{0} \in T_{m_0}M ~(\cong \mathbb{R}^{k}),$$ where $\|\cdot\|_{m_0}$ stands for the Riemannian norm induced by the Riemannian metric on $T_{m_0}M$, and its projection to the geodesic segment connecting $m_0$ and $p_0$ is zero; in other words,
\begin{equation}\label{eq:zero}
\int \rho'\{d(x,m_0)\}\cos(\beta_{x}) d\mathbb{P}_{X}(x)=0.
\end{equation}
Combining (\ref{eq:2ndderiv_final}) with (\ref{eq:zero}), we have shown that  (\ref{ineq:2nd_deriv2}) is positive.

 Since $\eta$ was arbitrarily chosen, the Hessian of $F$ at $m_0$ is positive definite. In conclusion, any critical point of the restricted map of $F$ on $B_{r}(p_0)$ must be a local minimizer.
\end{proof}

\end{appendix}

\bibliographystyle{imsart-number}
\bibliography{gmr}

@article{huckemann2011inference,
  title={Inference on 3{D} {P}rocrustes means: Tree bole growth, rank deficient diffusion tensors and perturbation models},
  author={Huckemann, Stephan F},
  journal={Scandinavian Journal of Statistics},
  volume={38},
  number={3},
  pages={424--446},
  year={2011},
  publisher={Wiley Online Library}
}

@article{jung2025averaging,
  title={Averaging symmetric positive-definite matrices on the space of eigen-decompositions},
  author={Jung, Sungkyu and Rooks, Brian and Groisser, David and Schwartzman, Armin},
  journal={Bernoulli},
  volume={31},
  number={2},
  pages={1552--1578},
  year={2025},
  publisher={Bernoulli Society for Mathematical Statistics and Probability}
}

@article{afsari2011riemannian,
  title={Riemannian ${L}^{p}$ center of mass: existence, uniqueness, and convexity},
  author={Afsari, Bijan},
  journal={Proceedings of the American Mathematical Society},
  volume={139},
  number={2},
  pages={655--673},
  year={2011}
}

@article{pennec2006riemannian,
  title={A {R}iemannian framework for tensor computing},
  author={Pennec, Xavier and Fillard, Pierre and Ayache, Nicholas},
  journal={International Journal of Computer Vision},
  volume={66},
  pages={41--66},
  year={2006},
  publisher={Springer}
}

@book{chavel2006riemannian,
  title={Riemannian Geometry: A modern Introduction},
  author={Chavel, Isaac},
  volume={98},
  year={2006},
  publisher={Cambridge University Press}
}

@article{jung2012analysis,
  title={Analysis of principal nested spheres},
  author={Jung, Sungkyu and Dryden, Ian L and Marron, James Stephen},
  journal={Biometrika},
  volume={99},
  number={3},
  pages={551--568},
  year={2012},
  publisher={Oxford University Press}
}

@article{park2026strong,
  title={Generalized {F}r\'{e}chet means with random minimizing domains and its strong consistency},
  author={Park, Jaesung and Jung, Sungkyu},
  journal={Biometrika},
  volume={113},
  pages = {asag002},
  year={2026}
}

@book{ley2017modern,
  title={Modern Directional Statistics},
  author={Ley, Christophe and Verdebout, Thomas},
  year={2017},
  publisher={CRC Press}
}

@article{petersen2019frechet,
  title={Fr{\'e}chet regression for random objects with {E}uclidean predictors},
  author={Petersen, Alexander and M{\"u}ller, Hans-Georg},
  journal={Annals of Statistics},
  volume={47},
  number={2},
  pages={691--719},
  year={2019},
  publisher={JSTOR}
}

@book{do1992riemannian,
  title={Riemannian Geometry},
  author={Do Carmo, Manfredo Perdigao},
  volume={2},
  year={1992},
  publisher={Springer}
}

@article{schotz2022strong,
  title={Strong laws of large numbers for generalizations of {F}r{\'e}chet mean sets},
  author={Sch{\"o}tz, Christof},
  journal={Statistics},
  volume={56},
  number={1},
  pages={34--52},
  year={2022},
  publisher={Taylor \& Francis}
}

@article{sinova2018m,
  title={{M}-estimators of location for functional data},
  author={Sinova, Beatriz and Gonz{\'a}lez-Rodr{\'\i}guez, Gil and Van Aelst, Stefan},
  journal={Bernoulli},
  pages={2328--2357},
  year={2018},
  publisher={JSTOR}
}

@article{ginestet2012weighted,
  title={Weighted {F}réchet means as convex combinations in metric spaces: properties and generalized median inequalities},
  author={Ginestet, Cedric E and Simmons, Andrew and Kolaczyk, Eric D},
  journal={Statistics \& Probability Letters},
  volume={82},
  number={10},
  pages={1859--1863},
  year={2012},
  publisher={Elsevier}
}

@article{casey2005positive,
  title={On positive derivatives and monotonicity},
  author={Casey, Stephen D and Holzsager, Richard},
  journal={Missouri Journal of Mathematical Sciences},
  volume={17},
  number={3},
  pages={161--173},
  year={2005},
  publisher={University of Central Missouri, School of Computer Science and Mathematics}
}

@article{lee2026huber,
  title={Huber means on Riemannian manifolds},
  author={Lee, Jongmin and Jung, Sungkyu},
  journal={Journal of the Royal Statistical Society Series B: Statistical Methodology},
  volume={88},
  number={2},
  pages={444--463},
  year={2026},
  publisher={Oxford University Press UK}
}

@article{de2021review,
  title={A review on robust {M}-estimators for regression analysis},
  author={De Menezes, DQF and Prata, Diego Martinez and Secchi, Argimiro R and Pinto, Jos{\'e} Carlos},
  journal={Computers \& Chemical Engineering},
  volume={147},
  pages={107254},
  year={2021},
  publisher={Elsevier}
}

@book{huber2009robust,
  title={Robust Statistics},
  author={Huber, Peter J and Ronchetti, Elvezio M},
  year={2009},
  publisher={John Wiley \& Sons Hoboken, NJ, USA}
}

@book{guillemin2010differential,
  title={Differential Topology},
  author={Guillemin, Victor and Pollack, Alan},
  volume={370},
  year={2010},
  publisher={American Mathematical Society}
}

@inproceedings{wang2025riemannian,
  title={Riemannian Change Point Detection on Manifolds with Robust Centroid Estimation},
  author={Wang, Xiuheng and Borsoi, Ricardo and Breloy, Arnaud and Richard, C{\'e}dric},
  booktitle={2025 33rd European Signal Processing Conference (EUSIPCO)},
  pages={1--5},
  year={2025},
  organization={IEEE}
}

@book{mardia2009directional,
  title={Directional statistics},
  author={Mardia, Kanti V and Jupp, Peter E},
  year={2009},
  publisher={John Wiley \& Sons}
}

\end{document}